\documentclass[a4paper, 12pt]{article}

\usepackage[a4paper]{geometry}
\usepackage[latin1]{inputenc}% Eingabekodierung, statt latin1 auch utf8 möglich
\usepackage[T1]{fontenc}% Schriftkodierung, wichtig!

\usepackage{graphicx} % einbinden von Grafiken mit \includegraphics
\usepackage{float}

\usepackage{xpatch,amsthm} % um Theoreme zu definieren, nur mit amsthm werden Kommentare in Klammern neben Theorem nicht fett
\usepackage{amsfonts} % für mathbb etc.
\usepackage[intlimits]{amsmath} % z.B. um Matheoperatoren zu definieren

\makeatletter 
   \xpatchcmd{\@thm}{\fontseries\mddefault\upshape}{}{}{} % same font as thm-header
\makeatother

\usepackage{IEEEtrantools} %für IEEEeqnarray\\

\usepackage{fancyhdr}
\usepackage{tikz,pgfplots}
\tikzset{>=stealth}

\usepackage{chngcntr}
\usepackage{wrapfig}
\usepackage[font={small,it},label font={bf,up}, margin=5mm]{caption}
\usepackage{sidecap}
\usepackage{subfig}

\newcommand{\N}{\mathbb{N}}    % Natuerliche Zahlen
\newcommand{\R}{\mathbb{R}}    % Reelle Zahlen
\newcommand {\apgt} {\ {\raise-.5ex\hbox{$\buildrel >\over\sim$}}\ }
\newcommand {\aplt} {\ {\raise-.5ex\hbox{$\buildrel<\over\sim$}}\ }

\newtheorem{theorem}{Theorem}\numberwithin{theorem}{section} 
\newtheorem{lemma}[theorem]{Lemma}

\newtheorem{definition}[theorem]{Definition}
\newtheorem{assumption}[theorem]{Assumption}
\newtheorem{corollary}[theorem]{Corollary}
\newtheorem{remark}[theorem]{Remark}

\newtheorem{notation}[theorem]{Notation}

\newtheoremstyle{break}   % Name
  {}                   % Space above
  {}                   % Space below
  {\itshape}              % Body font
  {0pt}                   % Indent amount
  {\bfseries}             % Theorem head font
  {.}                     % Punctuation after theorem head
  {\newline}              % Space after theorem head
  {}                      % Theorem head spec (can be left empty, meaning 'normal')

\theoremstyle{break}

\renewenvironment{proof}{{\bfseries Proof.}}{\qed}

\DeclareMathOperator*{\esssup}{ess\,sup}
\DeclareMathOperator*{\essinf}{ess\,inf}
\DeclareMathOperator*{\supp}{supp}
\DeclareMathOperator*{\diam}{diam}
\DeclareMathOperator*{\dist}{dist}

\DeclareMathOperator*{\Span}{span}
\DeclareMathOperator*{\Div}{div}
\DeclareMathOperator*{\interior}{int}
\DeclareMathOperator*{\argmax}{arg\,max}
\DeclareMathOperator*{\lb}{lb}
\DeclareMathOperator*{\vol}{vol}

\makeatletter
\def\cleardoublepage{\clearpage\if@twoside \ifodd\c@page\else\hbox{}
\thispagestyle{empty}
\newpage
\if@twocolumn\hbox{}\newpage\fi\fi\fi}
\makeatother

\usepackage[refpage]{nomencl}
\usepackage{xstring}

\begin{document}

\title{Adaptive Local (AL) Basis for Elliptic Problems with $L^\infty$-Coefficients}
\author{M. Weymuth\thanks{Institut f\"{u}r Mathematik, Universit\"{a}t Z\"{u}rich,
Winterthurerstrasse 190, CH-8057 Z\"{u}rich, Switzerland
}}
\maketitle

\begin{abstract}
\parindent0pt
We define a generalized finite element method for the discretization of elliptic partial
differential equations in heterogeneous media. In \cite{Sauter2012} a method has been introduced to set up an adaptive local finite element basis (AL basis) on a coarse mesh with mesh size $H$ which, typically, does not resolve the matrix of the media while the textbook finite element convergence rates are preserved. This method requires $O(\log(\frac{1}{H})^{d+1})$ basis functions per mesh point where $d$ denotes the spatial dimension of the computational domain. Since the continuous differential operator is involved in the construction, the method presented in \cite{Sauter2012} is only semidiscrete.
In this paper we present a fully discrete version of the method, where the AL basis is constructed by solving finite-dimensional localized problems.\\

\end{abstract}

\section{Introduction}
We consider second order elliptic partial differential equations with heterogeneous and highly varying (non-periodic) coefficients. 
Our emphasis is on the efficient numerical solution of problems whose coefficients contain a large number of different scales which we allow to be highly non-uniformly distributed over the domain. It is well-known that for such problems standard single scale numerical methods such as conventional finite element methods perform arbitrarily badly (see e.g. \cite{Babuska2000}).\\
Essentially there are two approaches to overcome this difficulty.  One is to design (non-polynomial) generalized finite element methods where the characteristic behaviour of the solution is incorporated in the shape of the basis functions. Early papers on this topic are \cite{Babuska1983, Babuska1994} which have been further developed e.g. in \cite{Hou1997, Hughes1998}. The second approach tries to simplify the coefficient by some approximation and then employs standard finite elements. Standard methods for simplifying the coefficients are based, e.g., on homogenization methods for periodic structures (see e.g., \cite{Jikov1994, Cioranescu1999, Bensoussan1978}), or on different upscaling techniques e.g. \cite{E2007, Repin2012}. In this paper we follow the first approach.\\
Many of the existing numerical methods belonging to the first approach show promising results in practice. However, their convergence analysis usually relies on certain structural assumptions on the coefficient (e.g. periodicity or scale separation).\\
In \cite{Babuska} a generalized finite element method for general $L^\infty$-coefficient is presented where the local finite element spaces are constructed via the solution of local eigenvalue problems. This approach is based on a partition of unity method (PUM, see e.g. \cite{Babuska1997, Melenk1995, Babuska1996}) and is closely related to our approach. Further approaches for the construction and analysis of a multiscale basis for problems with high contrast without structural assumptions on the coefficient include \cite{Larson2007, Ohwadi2011, Malqvist2011, Sauter2012}.\\
In \cite{Sauter2012} a generalized finite element space has been set up as the span of the adaptive local (AL) basis. It has been proved that on a regular finite element mesh with, possibly coarse, mesh size $H$ the number $p$ of basis functions per nodal point satisfies $p=O((\log\frac{1}{H})^{d+1})$. Moreover all basis functions have local support and the accuracy of the arising Galerkin finite element method with respect to the energy norm is of order $O(H)$ without any structural assumptions on the coefficient.\\
However, the method introduced in \cite{Sauter2012} is only semidiscrete since the inverse of the continuous solution operator $L$ is involved in the construction of the basis functions. In \cite{Weymuth2013}  this operator is replaced by a discrete operator $L_h$ which is obtained by a Galerkin discretization with a conforming finite-dimensional space $V_h$ on a sufficiently fine mesh. It is shown that the error estimates are preserved if the space $V_h$ satisfies the approximation property
\begin{equation*}
\sup_{f\in L^2(\Omega)\backslash\{0\}}\inf_{v\in V_h}\frac{\|L^{-1}f-v\|_{H^1(\Omega)}}{\|f\|_{L^2(\Omega)}}\leq C_{apx}H,
\end{equation*}
where $H$ denotes the coarse mesh width and the constant $C_{apx}$ is independent of $H$ and $f$. The operator $L_h^{-1}$ is a non-local fine-scale operator and the evaluation of its inverse is prohibitively expensive from the numerical point of view. In this paper we want to develop a localized version of the fully discrete method presented in \cite{Weymuth2013}.\\
The paper is structured as follows. In Section \ref{sec:problem} we formulate the model problem as well as the conditions on the coefficient. Section \ref{sec:method} is devoted to define the localized AL basis. In Section \ref{sec:regularity} we derive some $W^{1,p}$-regularity results for our model problem. These results are used in the error analysis. Finally in Section \ref{sec:error_analysis} the error analysis is presented.

\section{Model Problem}\label{sec:problem}
Let $\Omega\subset \R^d$, $d\in\{2,3\}$, be a bounded domain with $\partial\Omega\in C^1$. Let $\langle\cdot,\cdot\rangle$ denote the usual Euclidean scalar product on $\R^d$. The Sobolev space of real-valued functions in $L^2\left( \Omega\right)$ with gradients in $L^2(\Omega)$ and vanishing boundary trace is denoted by $H_0^1(\Omega)$ and its norm by $\|\cdot\|_{H^1(\Omega)}$.

We consider the following problem in variational form: Given $f\in L^{2}(\Omega)$, we are seeking $u\in H_0^1(\Omega)$ such that
\begin{equation}\label{prob_weak}
a(u,v):=\int\limits_\Omega\! \langle A\nabla u,\nabla v\rangle=\int_\Omega\! fv=: F(v)\quad\quad \forall\, v\in H_0^1(\Omega).
\end{equation}
The diffusion matrix $A\in L^{\infty}\left(\Omega,\R^{d\times d}_{sym}\right)$ is assumed to be uniformly elliptic, i.e.
\begin{equation}\label{coeff}
\begin{split}
0<\alpha(A,\Omega):=\essinf\limits_{x\in\Omega}\inf\limits_{v\in\R^d\backslash\{0\}}\frac{\langle A(x)v,v\rangle}{\langle v,v\rangle}\\
\infty>\beta(A,\Omega):=\esssup\limits_{x\in\Omega}\sup\limits_{v\in\R^d\backslash\{0\}}\frac{\langle A(x)v,v\rangle}{\langle v,v\rangle}.
\end{split}
\end{equation}
Since the bilinear form $a$ is symmetric, bounded and coercive, problem \eqref{prob_weak} has a unique solution.\\

We will discretize equation \eqref{prob_weak} with a conforming finite element method. For this let $\mathcal{G}$ be a conforming finite element mesh in the sense of Ciarlet \cite{Ciarlet} consisting of closed simplices $\tau$ which are the images of the reference element $\hat{\tau}$, i.e.\ the reference triangle (in 2d) or the reference tetrahedron (in 3d), under the element map $F_\tau\colon\hat{\tau}\to\tau$. We assume -- as is standard -- that the element maps of elements sharing an edge or a face induce the same parametrization on that edge or face. 
 Additionally, the element maps $F_\tau\colon \hat{\tau}\to\tau$ satisfy the following assumption.

\begin{assumption}\label{ass_quasi-uniform regular triangulation}
Each element map $F_\tau$ can be written as $F_\tau=R_\tau\circ A_\tau$, where $A_\tau$ is an affine map (corresponding to the scaling $\diam \tau$ of the simplex $\tau$) and $R_\tau$ is an analytic map which corresponds to the metric distortion at the possibly curved boundary and is independent of $\diam \tau$. Let $\tilde{\tau}:=A_\tau(\hat{\tau})$. The maps $R_\tau$ and $A_\tau$ satisfy for shape regularity constants $C_{affine}$, $C_{metric}$, $\gamma>0$ independent of $\diam \tau$:
\begin{align*}
&\|A_\tau'\|_{L^\infty(\hat{\tau})}\leq C_{affine}\diam \tau, && \|(A_\tau')^{-1}\|_{L^\infty(\tilde{\tau})}\leq C_{affine}(\diam \tau)^{-1}\\
&\|(R_\tau')^{-1}\|_{L^\infty(\tau)}\leq C_{metric}, && \|\nabla^nR_\tau\|_{L^\infty(\tilde{\tau})}\leq C_{metric}\gamma^nn!\quad \forall\, n\in\N_0.
\end{align*}
\end{assumption} 

The space of continuous, piecewise linear finite elements for the mesh $\mathcal{G}$ is given by
\begin{equation*}
S:=\left\{u\in H_0^1(\Omega) :~  u|_{\tau}\circ F_\tau \in \mathbb{P}_1\ \forall\, \tau \in \mathcal{G}\right\},
\end{equation*}
where $\mathbb{P}_1$ is the space of polynomials of degree $\leq 1$. Furthermore,
let $(b_i)_{i=1}^N$ denote the usual local nodal basis of $S$ (``hat functions''), i.e.\ $b_i(x_j)=\delta_{ij}$. We denote their support  by
\begin{equation*}
\omega_i:=\supp b_i.
\end{equation*}
Since $S\subset H_0^1(\Omega)$ is a finite-dimensional subspace, the abstract conforming Galerkin method to problem \eqref{prob_weak} can be formulated as: Find $u_S\in S$ such that
\begin{equation}\label{Galerkin problem}
a(u_S,v)=F(v)\quad \forall\, v\in S
\end{equation}
with $a(\cdot,\cdot)$ and $F(\cdot)$ as in \eqref{prob_weak}.

If the diffusion coefficient $A$, the right-hand side $f$ as well as the domain $\Omega$ of  \eqref{prob_weak} are sufficiently smooth such that the problem is  $H^2$-regular, then the unique solution $u_S$ of \eqref{Galerkin problem} satisfies the error estimate
\begin{equation*}
\|u-u_S\|_{H^1(\Omega)}\leq CH\|f\|_{L^2(\Omega)}
\end{equation*}
(see e.g. \cite{Ciarlet}). This estimate states linear convergence of the $\mathbb{P}_1$-finite element method as the mesh width $H$ tends to zero. However, the regularity assumption is not realistic for the problem class under consideration. It is well known that as long as the mesh $\mathcal{G}$ does not resolve the discontinuities and oscillations of $A$, the convergence rates of linear finite element methods are substantially reduced.\\

\section{The Adaptive Local (AL) Basis}\label{sec:method}
In this section we introduce a new generalized finite element method for the discretization of heterogeneous problems.

\subsection{Notation}
We assume that $\mathcal{G}$ is a conforming finite element mesh which is shape-regular and satisfies Assumption \ref{ass_quasi-uniform regular triangulation}. Moreover we suppose that the simplices $\tau\in\mathcal{G}$ are closed sets.

\begin{itemize}
\item [1)] Simplex layers around $\omega_i$ and corresponding meshes:\\
We define recursively
\begin{equation}\label{layers}
\begin{split}
\omega_{i,0}&:=\omega_i\\
\omega_{i,j+1}&:=\bigcup\left\{\tau :~ \tau\in\mathcal{G}  \ \text{and}\  \omega_{i,j}\cap \tau\neq\emptyset\right\},\quad j=0,1,2,\dots
\end{split}
\end{equation}

Finally, we set
\begin{equation*}
\mathcal{G}_{i,j}:=\left\{\tau\in\mathcal{G} :~ \tau\subset\omega_{i,j}\right\}.
\end{equation*}

\item [2)] Local neighbourhoods around the triangle patch $\omega_{i,1}$:\\
We set
\begin{equation}\label{omega_far}
\mathcal{G}_i^{far}:=\mathcal{G}_{i,2}\backslash\mathcal{G}_{i,1}\quad \text{and}\quad \omega_i^{far}:=\interior(\omega_{i,2}\backslash\omega_{i,1}).
\end{equation}

\item [3)] (Local) mesh width:\\
We set
\begin{equation}\label{local_mesh_width}
H_i:=\max_{\tau\in\mathcal{G}_{i,2}} \diam(\tau)\qquad\text{and}\qquad H:=\max_{1\leq i\leq N}H_i.
\end{equation}
Since the mesh is assumed to be shape-regular and the number of layers is bounded by 2, we can conclude that there exist positive constants $c$, $C$ and $C_\#$ such that
\begin{align}\label{diam_leq_CH}
&\min_{\tau\in\mathcal{G}_{i,2}}\rho_\tau\geq cH_i
&&\diam \omega_{i,2}\leq CH_i\\
&\dist(\omega_{i,1},\partial\omega_{i,2}\backslash\partial\Omega)=\delta_i\geq cH_i
&&\#\mathcal{G}_{i,2}\leq C_\#\nonumber
\end{align}

holds.
$\rho_\tau$ denotes the diameter of the maximal inscribed ball in $\tau$.
\item [4)] Refinement operator:\\ Let $\mathcal{T}^{macro}$ be a fixed triangulation (with possibly curved elements at the boundary) with element maps satisfying Assumption \ref{ass_quasi-uniform regular triangulation}. We introduce a refinement operator $\mathcal{R}^1(\cdot)$. The input is a conforming finite element mesh $\mathcal{T}$ where every element is marked for refinement and the output is a new conforming finite element mesh $\mathcal{R}^1(\mathcal{T})$. Recursively we define for $t\geq 2$ the iterated refinement operator 
\begin{equation}\label{refinement_op}\mathcal{R}^t(\mathcal{T}):=\mathcal{R}^1(\mathcal{R}^{t-1}(\mathcal{T})).
\end{equation}
\item [5)] Solution operator:\\
For a subdomain $D\subseteq\Omega$, let $L_D^{-1}\colon L^{2}(D)\to H_0^1(D)$ denote the solution operator associated with the (localized) variational form: Given $g\in L^{2}(D)$, find $w \in H_0^1(D)$ such that
\begin{equation}\label{local_sol_operator}
a_D(w,v):=\int_D\!\langle A\nabla w,\nabla v\rangle=\int_D\! gv=:G(v)\quad \forall\, v\in H_0^1(D).
\end{equation}
\end{itemize}

\begin{remark}\label{rem_overlap}
Note that the patches $\omega_{i,j}$, $0\leq j\leq 2$, have finite overlap. For every $\tau\in\mathcal{G}$ there exists $m_{\tau,j}\in\N$ such that
\begin{equation*}
\#\{i :~ \tau\in \omega_{i,j}\}= m_{\tau,j},\quad 0\leq j\leq 2.
\end{equation*}
We set
\begin{equation}\label{overlap_const}
M_j:=\max_\tau m_{\tau,j},\quad 0\leq j\leq 2.
\end{equation}
\end{remark}

\subsection{Construction of the Local Approximation Spaces}

On each patch $\omega_{i,2}$ ($1\leq i\leq N$) we will set up two low-dimensional local approximation spaces called $V_i^{near}$ and $V_i^{far}$. In order to get $V_i^{far}$ we will first construct an intermediate space $X_i^{far}$ which is the high-dimensional space of locally $L$-harmonic functions and can be approximated by a low-dimensional space.\\

We fix $i\in\mathcal{I}:=\{1,\dots,N\}$. The construction of the space $V_i^{near}$ respectively $X_i^{far}$ goes as follows.
We set
\begin{equation}\label{piecewise_constant}
S_0(\mathcal{G}):=\Span\left\{\chi_\tau :~ \tau\in\mathcal{G}\right\},
\end{equation}
where $\chi_\tau\colon \Omega\to\R$ is the characteristic function for the simplex $\tau\in\mathcal{G}$ and $H$ is the global mesh width. $S_0(\mathcal{G})$ is the space of piecewise constant functions on $\mathcal{G}$. Furthermore we define the space
\begin{equation}\label{S_i2h}
S_{i,2}:=\{u|_{\omega_{i,2}} :~ u\in S_{fine}\wedge \supp u\subset\omega_{i,2}\},
\end{equation}
where $S_{fine}$ is some finite-dimensional fine-scale space satisfying
\begin{equation*}
S_{fine}\subset H^1(\Omega).
\end{equation*}

The local approximation spaces are constructed by solving conventional finite element problems.  For the nearfield part, i.e.\ $\tau\in\mathcal{G}_{i,1}$, we want to find $\tilde{B}_{i,\tau}^{near}\in S_{i,2}$ such that
\begin{equation}\label{nearfield_sol}
\int_{\omega_{i,2}}\!\langle A\nabla\tilde{B}_{i,\tau}^{near},\nabla v\rangle=\int_{\omega_{i,2}}\!\chi_\tau v\quad\forall\, v\in S_{i,2}.
\end{equation}
Then we set
\begin{equation*}
B_{i,\tau}^{near}:=b_i\tilde{B}_{i,\tau}^{near}
\end{equation*}
and finally our local approximation space for the nearfield part can be defined as
\begin{equation*}
V_i^{near}:=\Span\{B_{i,\tau}^{near} :~ \tau\in\mathcal{G}_{i,1}\}.
\end{equation*}
The construction of the local approximation space for the farfield part can be done analogously, but the error analysis shows that for preserving the linear convergence rate of the method we have to refine the mesh $\mathcal{G}_i^{far}$. Thus for $\tau\in\mathcal{R}^t(\mathcal{G}_i^{far})$ we are seeking $\tilde{B}_{i,\tau}^{far}\in S_{i,2}$ such that
\begin{equation*}
\int_{\omega_{i,2}}\!\langle A\nabla \tilde{B}_{i,\tau}^{far},\nabla v\rangle=\int_{\omega_{i,2}}\!\chi_\tau v\quad \forall\, v\in S_{i,2}.
\end{equation*}
The error analysis will show that the refinement parameter $t$ has to be chosen as $t=\lceil\lb\frac{1}{H_i}\rceil$. We set
\begin{equation*}
X_i^{far}:=\Span\{\tilde{B}_{i,\tau}^{far}|_{\omega_{i,1}} :~ \tau\in\mathcal{R}^t(\mathcal{G}_i^{far})\}.
\end{equation*}

\begin{remark}~
\begin{itemize}
\item[a)] In order to get a linear convergence rate in the $H^1$-norm the space $S_{fine}$ in \eqref{S_i2h} has to be chosen such that
\begin{equation*}
\sup\limits_{f\in L^{2}(\omega_{i,2})\backslash\{0\}}\inf\limits_{v\in S_{i,2}}\frac{\left\|L_{\omega_{i,2}}^{-1}f-v\right\|_{H^1(\omega_{i,2})}}{\left\|f\right\|_{L^2(\omega_{i,2})}}\leq C_{apx}H_i^2
\end{equation*}
holds, where  the constant $C_{apx}$ is independent of $H_i$ and $f$.
\item[b)] The functions in $X_i^{far}$ are locally $L$-harmonic on $\interior(\omega_{i,1})$, i.e.\ any $v\in X_i^{far}$ satisfies
\begin{equation*}
\int_{\omega_{i,1}}\!\langle A\nabla v,\nabla w\rangle=0\quad \forall\, w\in S_{i,1}:=\{w|_{\omega_{i,1}} :~ w\in S_{fine}\wedge \supp w\subset\omega_{i,1}\}.
\end{equation*}
\end{itemize}
\end{remark}

\subsection{Approximation of $\boldsymbol{X_i^{far}}$}

Our goal is to approximate the space $X_i^{far}$ by a low-dimensional space $V_i^{far}$. The construction of this approximation is based on results in \cite{Bebendorf, Boerm, Sauter2012}. \\

Let $\omega_i$, $\omega_{i,1}$ as in \eqref{layers} and assume that $\omega_i\cap\partial\Omega=\emptyset$.
We introduce intermediate layers between $\omega_i$ and $\omega_{i,1}$. Therefore we set $r_{i,1}:=\dist(\omega_i,\partial\omega_{i,1})$ and
\begin{equation}\label{def_rij}
r_{i,j}:=\left(1-\frac{j-1}{\ell-1}\right)r_{i,1},\quad\quad 2\leq j\leq \ell,
\end{equation}
where $\ell$ will be fixed later. It holds $r_{i,1}>r_{i,2}>\dots>r_{i,\ell}=0$. The intermediate layers are given by
\begin{align*}
D_{i,0}&:= \omega_{i,1}\\
D_{i,j}&:=\left\{x\in\omega_{i,1} :~ \dist(x,\omega_i)\leq r_{i,j}\right\},\quad\quad 1\leq j\leq \ell,
\end{align*}
and satisfy $\omega_i=D_{i,\ell}\subset D_{i,\ell -1}\subset\cdots \subset D_{i,1} \subset D_{i,0}= \omega_{i,1}$. Note that if $\omega_i$ and $\omega_{i,1}$ are convex, then also  the domains $D_{i,j}$ are convex for all $0\leq j\leq\ell$. In \cite{Bebendorf} it is shown that for any $\kappa_j\in\N$ there exists a subspace $V_{\kappa_j}\subset X(D_{i,j})$ such that $\dim V_{\kappa_j}\leq \kappa_j$ and the estimate
\begin{equation}\label{eq_harmonic}
\inf_{v\in V_{\kappa_j}}\|u-v\|_{L^2(D_{i,j})}\leq C\frac{\diam(D_{i,j})}{\sqrt[d]{\kappa_j}}\|\nabla u\|_{L^2(D_{i,j})}
\end{equation} 
is satisfied.\footnote{$X(D_{i,j})$ denotes the space of locally harmonic functions on $D_{i,j}$. Note that the constant $C$ in \eqref{eq_harmonic} depends on Poincar\'{e}'s constant and hence on the shape of $D_{i,j}$. If $D_{i,j}$ is convex, then $C=2\sqrt[d]{2}/\pi$ (cf. \cite{Bebendorf}).} In order to construct these subspaces $V_{\kappa_j}=:\tilde{V}_{i,j}^{far}$ for $1\leq j\leq \ell $ we use $L^2$-orthogonal projections onto $X_i^{far}$.

 We set $\kappa_j=:k^d$, where $k\in\N$ will be fixed later. 
 For $\rho>0$ let $\mathcal{G}_\rho$ denote a Cartesian tensor mesh on $\R^d$, $d\in\{2,3\}$, which consists of $d$-dimensional elements with side length $\rho$. Then define
\begin{equation*}
\mathcal{\tilde{G}}_{i,j}:=\left\{D_{i,j}\cap\tau :~ \tau\in\mathcal{G}_\rho\ \text{with}\ \rho:=\frac{\diam(D_{i,j})}{k}\right\},\quad 1\leq j\leq \ell
\end{equation*}
 and

\begin{equation*}
\tilde{V}_{i,j}^{far}:=\Span \left\{(\mathcal{P}_i\chi_t)|_{\omega_i} :~ t\in \mathcal{\tilde{G}}_{i,j}\right\},
\end{equation*}
where $\mathcal{P}_i\colon L^2(\omega_{i,2})\to X_i^{far}$ is the $L^2$-orthogonal projection. We set

\begin{equation}\label{tilde_V_far}
\tilde{V}_i^{far}:=\tilde{V}_{i,1}^{far}+\tilde{V}_{i,2}^{far}+\cdots +\tilde{V}_{i,\ell}^{far}
\end{equation}
and finally,

\begin{equation*}
V_i^{far}:=\left\{b_iv :~ v\in \tilde{V}_i^{far}\right\}.
\end{equation*}

\begin{remark}
If $\omega_i\cap\partial\Omega\neq\emptyset$ we have to make the following small modifications. We set $r_{i,1}:=\dist(\omega_i,\partial\omega_{i,1}\backslash\partial\Omega)$ and $r_{i,j}$ is defined as in \eqref{def_rij}. The intermediate layers are given by
\begin{align*}
D_{i,0}&:= \omega_{i,1}\cup\{x\in\R^d :~  \dist(x,\omega_i)\leq r_{i,1}\}\\
D_{i,j}&:=\left\{x\in\R^d :~ \dist(x,\omega_i)\leq r_{i,j}\right\},\quad\quad 1\leq j\leq \ell.
\end{align*}
The remaining part of the construction is exactly the same as above.
\end{remark}

\subsection{Definition of the AL Basis}
\begin{remark}\label{remdim}
Since $b_i\in W_0^{1,\infty}(\omega_i)$ and $X_i^{far}\subset H^1(\omega_{i,1})$ we conclude that $b_iv\in H_0^1(\omega_{i})$ for all $v\in \tilde{V}_i^{far}$. Thus we can identify $b_iv$ by its extension by  zero to a function (again denoted by $b_iv$) in $H_0^1(\Omega)$. In this sense we have
\begin{equation*}
V_i^{far}\subset H_0^1(\Omega),\quad\quad \dim V_i^{far}\leq  \sum\limits_{j=1}^{\ell} \#\mathcal{\tilde{G}}_{i,j}\leq \sum\limits_{j=1}^{\ell} k^d=\ell k^d.
\end{equation*}
\end{remark}

\begin{definition}[AL basis]
For any support $\omega_i$ the set of \emph{AL basis} functions consists of 
\begin{equation*}
V_i^{near}:=\Span\left\{b_i\tilde{B}_{i,\tau}^{near} :~ \tau\in\mathcal{G}_{i,1}\right\}
\end{equation*}
where $\tilde{B}_{i,\tau}^{near}$ is the solution of problem \eqref{nearfield_sol} and of
\begin{equation*}
V_i^{far}:=\left\{b_iv :~ v\in \tilde{V}_i^{far}\right\}.
\end{equation*}
The general notation is $b_{i,j}$, $1\leq j\leq s_i$, $1\leq i\leq N$, where $s_i:=\dim(V_i^{near}+V_i^{far})$.
The corresponding generalized finite element space $V_{AL}$ is given by
\begin{equation}\label{def:VAL-II}
V_{AL}:=\left(V_1^{near}+V_1^{far}\right)+\left(V_2^{near}+V_2^{far}\right)+\dots +\left(V_N^{near}+V_N^{far}\right).
\end{equation}
\end{definition}

The Galerkin discretization for the generalized finite element space $V_{AL}$ is given by seeking $u_{AL}^{GAL}\in V_{AL}$ such that
\begin{equation}\label{Galerkin_sol_II}
a(u_{AL}^{GAL}, v)=F(v)\quad \forall\, v\in V_{AL}.
\end{equation}

Problem \eqref{Galerkin_sol_II} has a unique solution and  is equivalent to a system of linear equations of the form
\begin{equation}\label{eq_linear_system}
\sum_{i=1}^N\sum_{j=1}^{s_i}a(b_{k,\ell},b_{i,j})c_{i,j}=F(b_{k,\ell}),\quad 1\leq \ell\leq s_i,\; 1\leq k\leq N
\end{equation} 
or 
\begin{equation*}
Bc=F
\end{equation*}
where $B$ is the stiffness matrix, whose elements are 
\begin{equation*}
B(\ell,k;j,i):=a(b_{k,\ell},b_{i,j})=\int_{\omega_i\cap\omega_k}\!\langle A\nabla b_{k,\ell},\nabla b_{i,j}\rangle
\end{equation*}
and $F$ is the load vector which is defined as 
\begin{equation*}
F(k;\ell):=\int_{\omega_k}\!f b_{k,\ell}. 
\end{equation*}
If $c:=\{c_{i,j}\}$ is a solution of \eqref{eq_linear_system}, then $u_{AL}^{GAL}$ can be written as 
\begin{equation*}
u_{AL}^{GAL}=\sum_{i=1}^N\sum_{j=1}^{s_i}c_{i,j}b_{i,j}.
\end{equation*}

\section{$\boldsymbol{W^{1,p}}$-Regularity of the Poisson Problem with $\boldsymbol{L^\infty}$-Coefficient}\label{sec:regularity}

Let $u$ be the solution of \eqref{prob_weak}. Our goal is to derive $L^p(\Omega)$-regularity estimates for the gradient of $u$
for some $p>2$. We start from a Laplace problem, i.e. the coefficient $A$ is equal to the identity matrix and employ then a perturbation argument in order to get the desired estimates for a uniformly elliptic diffusion matrix $A\in L^\infty(\Omega,\R_{sym}^{d\times d})$. We will see that our estimates only depend on the size of the jumps in the coefficient. We consider the following problem: Find  $w\in H_0^1(\Omega)$ such that
\begin{equation}\label{laplace}
\int_\Omega\!\langle\nabla w,\nabla v\rangle=F(v)\quad\forall\, v\in H_0^1(\Omega).
\end{equation} 

\begin{theorem}[\cite{Simader1996}]\label{theo_Simader}
Let $\Omega\subset\R^d$, $d\geq2$, be a bounded
domain with $\partial\Omega\in C^1$. Let $1<p<\infty$. Then, for every $F\in W^{-1,p}(\Omega)$, problem \eqref{laplace} has a unique solution $w\in W_0^{1,p}(\Omega)$ which satisfies
\begin{equation*}
K_p^{-1}\|\nabla w\|_{L^p(\Omega)} \leq\| F\|_{W^{-1,p}(\Omega)}\leq \|\nabla w\|_{L^p(\Omega)} 
\end{equation*}
with the Laplace $W^{1,p}$-regularity constant $K_{p}$ and
\begin{equation*}
\|F\|_{W^{-1,p}(\Omega)}:=\sup_{\substack{v\in W_0^{1,p'}(\Omega)\\ \|v\|_{W^{1,p'}(\Omega)}\leq 1}}\left|\int_\Omega\langle\nabla w,\nabla v\rangle\right|.
\end{equation*}
\end{theorem}

\begin{remark}
The constant $K_p$ is independent of $F$ (and $w$) but depends on $\Omega$, $d$ and $p$. We have $K_2=1$ and, for $p>2$, $K_p$ is
non-decreasing and continuous in $p$ (cf. \cite{Meyers1963}).
\end{remark}

\begin{figure}[h]
\centering
\subfloat{\includegraphics[scale=0.6]{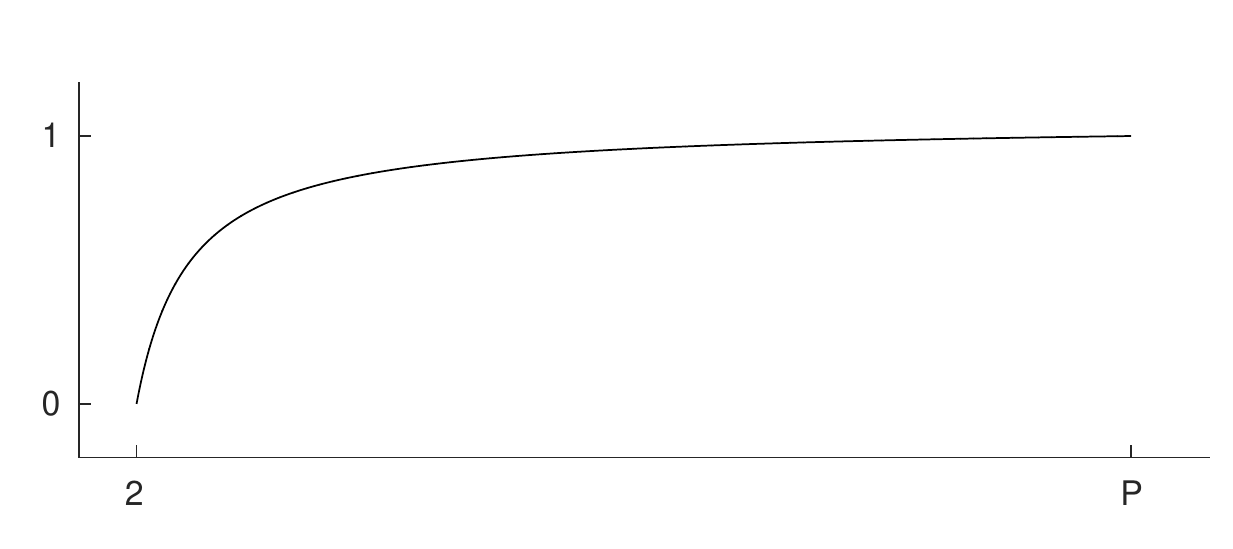}}
\subfloat{\includegraphics[scale=0.6]{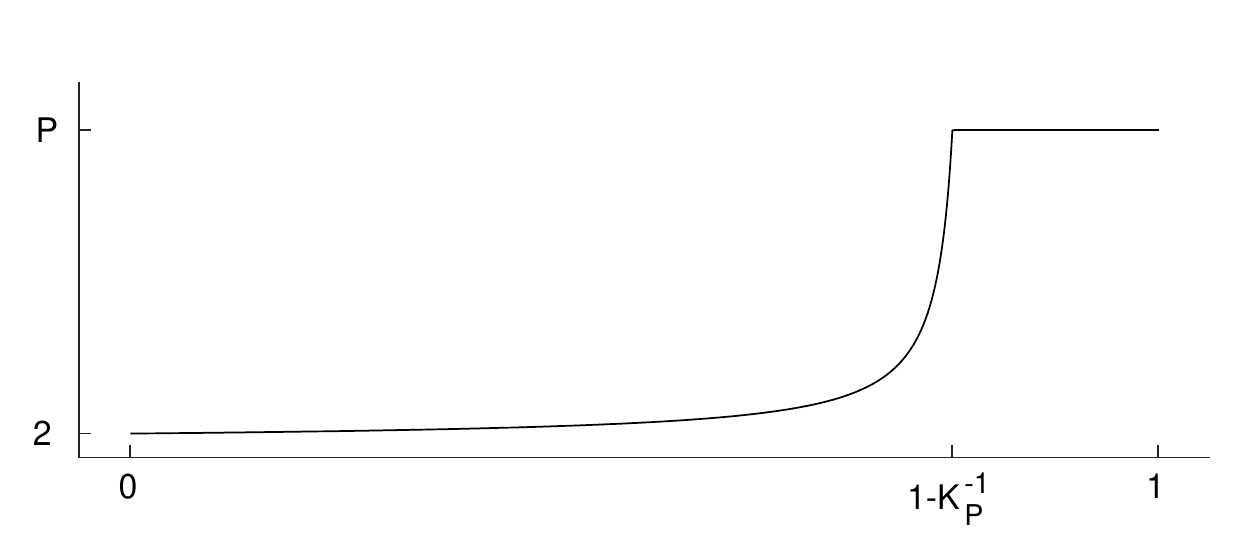}}
\caption{The function $\eta(p)$ (left) and the function $p^{*}(t)$ (right).}\label{fig:p_star}
\end{figure}

Let $2<P<\infty$ be fixed. We define
\begin{equation*}
\eta(p):=\frac{1/2-1/p}{1/2-1/P}\qquad 2\leq p\leq P.
\end{equation*}

It can be seen in Figure \ref{fig:p_star} that $\eta(p)$ increases from the value zero at $p=2$ to the value one at $p=P$. Furthermore for any $t\in[0,1]$, we set

\begin{equation}\label{qstern}
p^*(t):=\argmax\left\{K_P^{-\eta(p)}\geq 1-t :~ 2\leq p\leq P\right\}.
\end{equation}

The function $K_P^{-\eta(p)}$ decreases from the value $1$ at $p=2$ to the value $1/K_P$ at $p=P$. The function $p^*(t)$ takes the value $2$ at $t=0$, increases then to the value $P$ at $t=1-1/K_P$ and remains constant for $t\in[1-1/K_P,1]$ (see Figure \ref{fig:p_star}).

\begin{theorem}\label{theo_gradient_estimate}
Let $\Omega\subset\R^d$, $d\geq 2$, be a bounded domain and let $\partial\Omega\in C^1$. If $A\in L^\infty\left(\Omega,\R^{d\times d}_{sym}\right)$ satisfies \eqref{coeff} and $F\in W^{-1,P}(\Omega)$ for some $P>2$, then for the solution $u\in H_0^1(\Omega)$ of $\eqref{prob_weak}$ the estimate
\begin{equation*}
\|\nabla u\|_{L^p(\Omega)}\leq C\|F\|_{W^{-1,p}(\Omega)}
\end{equation*}
holds provided $2\leq p< p^*(\alpha/\beta)$ with $p^*$ as in \eqref{qstern} and $C:=\frac{1}{\beta}\frac{K_P^{\eta(p)}}{1-K_P^{\eta(p)}\left(1-\frac{\alpha}{\beta}\right)}$.
\end{theorem}

For a proof we refer to \cite{Meyers1963, Nochetto2013, Weymuth2016}.

\begin{remark}
Let $P\in (2,\infty)$ be fixed and $K_P$ as in Theorem \ref{theo_Simader}. If the coefficient $A$ is such that $\alpha/\beta\in[1-1/K_P,1]$ and $F\in W^{-1,P}(\Omega)$, then the solution of \eqref{prob_weak} satisfies the estimate 
\begin{equation*}
\|\nabla u\|_{L^p(\Omega)}\leq C\|F\|_{W^{-1,p}(\Omega)},\quad C=\frac{1}{\beta}\frac{K_P^{\eta(p)}}{1-K_P^{\eta(p)}\left(1-\frac{\alpha}{\beta}\right)}
\end{equation*}
 for $2\leq p< P=p^*(\alpha/\beta)$ with $p^*$ as in \eqref{qstern}. This is due to the fact that the function $p^*$ takes the value $P$ at $1-1/K_P$ and remains constant in the interval $[1-1/K_P,1]$ (cf.\ Figure \ref{fig:p_star}).
 \end{remark}
 
Note that for a given coefficient $A\in L^\infty\left(\Omega,\R^{d\times d}_{sym}\right)$ one can always determine a $P>2$ such that $\alpha/\beta\in [1-1/K_P, 1]$. The $P$ depends only on the size of the jumps in the coefficient. For constant coefficients $P$ can be chosen arbitrarily close to infinity, whereas for coefficients with large  jumps $P$ is close to 2.

\section{Error Analysis}\label{sec:error_analysis}
This section analyzes the generalized finite element method which has been introduced in Section \ref{sec:method}. It is based on results in \cite{Bebendorf, Boerm, Sauter2012}.\\
 
The norm in $L^p(\Omega)$ will be denoted by $\|\cdot\|_{L^p(\Omega)}$. We always use the notation that, for $p\in\left[1,\infty\right]$, the number $p'\in\left[ 1,\infty\right] $ is defined via $\frac{1}{p}+\frac{1}{p'}=1$. Further we will need the Sobolev
space $W^{1,p}(\Omega)$ consisting of functions in $L^p\left(\Omega\right)$ with gradients in $L^p(\Omega)$. Its standard norm is
denoted by $\|\cdot\|_{W^{1,p}(\Omega)}$. The space of functions denoted by $W_0^{1,p}(\Omega)$ is the closure of $C_0^\infty(\Omega)$ with respect to the norm $\|\cdot\|_{W^{1,p}(\Omega)}$. We also use the space $W^{-1,p}\left(\Omega\right):=(W_0^{1,p'}\left(\Omega\right))'$ endowed with the standard dual norm $\|\cdot\|
_{W^{-1,p}(\Omega)}$. For vector and matrix valued functions, we use the same notation for the Lebesgue and Sobolev spaces as well as for the corresponding norms. For functions in $L^{2}\left(  \Omega,\mathbb{R}^{d}\right)  $
we set
\begin{equation*}
\|\cdot\|_{L^p(\Omega}):=\|\,\|\cdot\|_{\ell^{p}}\,\|_{L^p(\Omega)},
\end{equation*}
where $\|\cdot\|_{\ell^p}$ denotes the discrete $\ell^p$-norm in $\R^d$.\\

For the error analysis it is supposed that the following assumption holds.
\begin{assumption}\label{ass}
\begin{equation*}
\sup\limits_{f\in L^{2}(\omega_{i,2})\backslash\{0\}}\inf\limits_{v\in S_{i,2}}\frac{\left\|L_{\omega_{i,2}}^{-1}f-v\right\|_{H^1(\omega_{i,2})}}{\left\|f\right\|_{L^2(\omega_{i,2})}}\leq C_{apx}H_i^2,
\end{equation*}
where $S_{i,2}$ is as defined in \eqref{S_i2h}, $H_i$ is the mesh width of  $\mathcal{G}_{i,2}$ (cf.\ \eqref{omega_far} and \eqref{local_mesh_width}) and the constant $C_{apx}$ is independent of $H_i$ and $f$.
\end{assumption}

\begin{notation}
Let $\tilde{L}_{\omega_{i,2}}^{-1}: L^2(\interior(\omega_{i,2}))\to S_{i,2}$ denote the discrete local solution operator: Given $g\in L^2(\interior(\omega_{i,2}))$ find $\tilde{B}_{i,\tau}\in S_{i,2}$ such that
\begin{equation*}
\int_{\omega_{i,2}}\!\langle A\nabla\tilde{B}_{i,\tau},\nabla v\rangle=\int_{\omega_{i,2}}\!g v\quad\forall\, v\in S_{i,2}.
\end{equation*}
\end{notation}

\begin{corollary}
C\'ea's lemma and Assumption \ref{ass} imply
\begin{equation}\label{assumption}
\left\|L_{\omega_{i,2}}^{-1}f-\tilde{L}_{\omega_{i,2}}^{-1}f\right\|_{H^1(\omega_{i,2})}\leq \frac{\beta}{\alpha }CH_i^2\left\|f\right\|_{L^2(\omega_{i,2})},
\end{equation}
where $\alpha,\,\beta$ are the constants from \eqref{coeff}. $C$ depends on $C_{apx}$ and on Friedrichs' constant.
\end{corollary}

\begin{proof}
By C\'ea's lemma we get
\begin{equation*}
\left\|L_{\omega_{i,2}}^{-1}f-\tilde{L}_{\omega_{i,2}}^{-1}f\right\|_{H^1(\omega_{i,2})}\leq \frac{\beta}{\alpha} C \inf_{v\in S_{i,2}}\left\|L_{\omega_{i,2}}^{-1}f-v\right\|_{H^1(\omega_{i,2})}.
\end{equation*}
Assumption \ref{ass} implies
\begin{align*}
\left\|L_{\omega_{i,2}}^{-1}f-\tilde{L}_{\omega_{i,2}}^{-1}f\right\|_{H^1(\omega_{i,2})}&\leq \frac{\beta}{\alpha}C\sup_{f\in L^2(\omega_{i,2})\backslash\{0\}} \inf_{v\in S_{i,2}}\left\|L_{\omega_{i,2}}^{-1}f-v\right\|_{H^1(\omega_{i,2})}\\
&\leq \frac{\beta}{\alpha}C H_i^2\|f\|_{L^2(\omega_{i,2})}
\end{align*}
with a constant $C$ depending on Friedrichs' constant and $C_{apx}$.
\end{proof}

\begin{remark}
The ellipticity of $L_{\omega_{i,2}}^{-1}$, the assumption \eqref{coeff} on the coefficient $A$, and the conformity of the finite element space $S_{i,2}$ imply that the approximation $\tilde{L}_{\omega_{i,2}}^{-1}$ is elliptic and
\begin{equation}\label{L}
\left\|\tilde{L}_{\omega_{i,2}}^{-1}\right\|_{H_0^1(\omega_{i,2})\leftarrow H^{-1}(\omega_{i,2})}\leq \frac{C}{\alpha},
\end{equation}
where $\alpha$ is defined in \eqref{coeff}.
\end{remark}

\begin{lemma}\label{lemma_projection}
Let $\mathcal{G}$ be a conforming finite element mesh which satisfies Assumption \ref{ass_quasi-uniform regular triangulation}. Further let $g_i\in L^2(\omega_{i,2})$ and denote by $P_i$ the $L^2$-orthogonal projection of $L^2(\omega_{i,2})$ onto $S_0(\mathcal{G})$ (cf.\ \eqref{piecewise_constant}). Then
\begin{equation*}
\|g_i-P_ig_i\|_{H^{-1}(\omega_{i,2})}\leq CH\|g_i\|_{L^2(\omega_{i,2})},
\end{equation*}
where $H$ denotes the mesh width of $\mathcal{G}$.
\end{lemma}

\begin{proof}
Using H\"{o}lder's inequality and Friedrichs' inequality we get for $\psi_i\in H_0^1(\omega_{i,2})$
\begin{align}\label{proj_est}
|\langle g_i-P_ig_i, \psi_i\rangle_{L^2(\omega_{i,2})}|&\leq \|g_i-P_ig_i\|_{L^2(\omega_{i,2})}\|\psi_i\|_{L^2(\omega_{i,2})}\nonumber\\&\leq C H \|g_i\|_{L^2(\omega_{i,2})}\|\psi_i\|_{H^1(\omega_{i,2})}.
\end{align}
By the definition of the $H^{-1}$-norm and \eqref{proj_est} we obtain \belowdisplayskip=-12pt
\begin{equation*}
\|g_i-P_ig_i\|_{H^{-1}(\omega_{i,2})}=\sup_{v\in H_0^1(\omega_{i,2})}\frac{|\langle g_i-P_ig_i, v\rangle_{L^2(\omega_{i,2})}|}{\|v\|_{H^1(\omega_{i,2})}}\leq CH\|g_i\|_{L^2(\omega_{i,2})}.\qedhere
\end{equation*}
\end{proof}

\vspace{0.5cm} 

We fix some $Q\in (2,\infty)$. For $1\leq i\leq N$ let $\chi_i\colon \Omega\to\R$ be a cutoff function satisfying $\chi_i|_{\omega_{i,1}}=1$ and $\chi_i|_{\Omega\backslash\omega_{i,2}}=0$. Morover the following properties are fulfilled for $\frac{Q'}{3}< q< \frac{Q}{3}$.
\begin{align}
\|\chi_i\|_{L^q(\omega_i^{far})}&\leq CH_i^{\frac{d}{q}}\label{cutoff_estimate_1}\\
\|\nabla\chi_i\|_{L^q(\omega_i^{far})}&\leq CH_i^{\frac{d}{q}-1}\label{cutoff_estimate_2}\\
\|\Div(A\nabla\chi_i)\|_{L^q(\omega_i^{far})}&\leq CH_i^{\frac{d}{q}-2}\label{cutoff_estimate_3}
\end{align}

\begin{remark}
For the explicit construction of $\chi_i$ we refer to \cite{Weymuth2016}. The cutoff functions are constructed by solving homogeneous Dirichlet problems. It would be desirable to have $\chi_i\in W^{1,\infty}(\Omega)$ as well as $\Div(A\nabla\chi_i)\in L^\infty(\Omega)$. However, for $d\geq 2$ this is not possible.
\end{remark}
Let $f\in L^{p}(\Omega)$ ($p\in[2,\infty]$) be given and define $u:=L_\Omega^{-1}f$. 
We set $u_i:=\chi_i(u-\bar{u}_i)$ with
\begin{equation*}
\bar{u}_i:=\begin{cases}
\frac{1}{\vol(\omega_{i}^{far})}\int\limits_{\omega_i^{far}}\!u &\text{in}\ \omega_i^{far}\\
0 &\text{otherwise}.
\end{cases}
\end{equation*}
We observe that
\begin{equation*}
u_i=L_{\omega_{i,2}}^{-1}\left(g_i\right)
\end{equation*}
with
\begin{equation*}
g_i=\begin{cases}
f &\text{in}\ \omega_{i,1}\\
\chi_if-2\langle A\nabla \chi_i,\nabla u\rangle-(u-\bar{u}_i)\Div\left(A\nabla\chi_i\right) &\text{in}\ \omega_i^{far}.
\end{cases}
\end{equation*}

We set
\begin{equation}\label{g_i_near}
g_i^{near}:=\begin{cases}
f &\text{in}\ \omega_{i,1}\\
0 & \text{in}\ \omega_i^{far}
\end{cases}
\end{equation}
 and
\begin{equation}\label{g_i_far}
g_i^{far}:=\begin{cases} 0 &\text{in}\ \omega_{i,1}\\
\chi_if-2\langle A\nabla \chi_i,\nabla u\rangle-(u-\bar{u}_i)\Div\left(A\nabla\chi_i\right) &\text{in}\ \omega_i^{far}.
\end{cases}
\end{equation}

This allows us to introduce
\begin{equation}\label{splitting}
u_i=u_i^{near}+u_i^{far}:=L_{\omega_{i,2}}^{-1}\left(g_i^{near}\right)+L_{\omega_{i,2}}^{-1}(g_i^{far}).
\end{equation}
Define
\begin{equation}\label{approx_near_far}
\tilde{u}_i^{near}:=\tilde{L}_{\omega_{i,2}}^{-1}(P_ig_i^{near})\qquad \text{and}\qquad
\tilde{u}_i^{far}:=\tilde{L}_{\omega_{i,2}}^{-1}(P_i^tg_i^{far}),
\end{equation}
where $P_i$ denotes the $L^2$-orthogonal projection of $L^2(\omega_{i,1})$ onto $S_0(\mathcal{G}_{i,1})$ (cf.\ \eqref{piecewise_constant}) and $P_i^t$ is the $L^2$-orthogonal projection of $L^2(\omega_i^{far})$ onto  $S_0(\mathcal{R}^t(\mathcal{G}_i^{far}))$ which is the space of piecewise constant functions on the $t$-times refined mesh ($t$ will be fixed later).\\

The following lemma is a slight modification of a result presented in \cite{Boerm, Sauter2012}.

\begin{lemma}\label{lemma_aprox_tilde_u_far}
Let $\tilde{u}_i^{far}$ as in \eqref{approx_near_far} and $\tilde{V}_i^{far}$ as in \eqref{tilde_V_far}. There exists $\hat{u}_i^{far}\in\tilde{V}_i^{far}$ such that
\begin{equation*}
\|\tilde{u}_i^{far}-\hat{u}_i^{far}\|_{H^m(\omega_i)}\leq CH_i^{3-m}\|\nabla\tilde{u}_i^{far}\|_{L^2(\omega_{i,1})}\quad m=0,1
\end{equation*}
with $H_i$ as in \eqref{local_mesh_width}.
\end{lemma}

\begin{proof}
Set
\begin{equation*}
\ell:=\max \left\{2,\left\lceil\frac{2}{\log 2}\log\frac{1}{H_i}\right\rceil\right\}\quad \text{and}\quad k:=\left\lceil\frac{2c_0\ell^2}{(\ell-1)}\right\rceil
\end{equation*}
for some $c_0=O(1).$ Choosing $p\leftarrow \ell$, $\ell\leftarrow k$, $i\leftarrow \ell$, $c\leftarrow c_0$, and $\delta\leftarrow O(H_i)$ in the second estimate of \cite[p.\,172]{Boerm} yields
\begin{equation}\label{Boerm1}
\|\tilde{u}_i^{far}-\hat{u}_i^{far}\|_{L^2(\omega_i)}\leq C H_i \left(c_0\frac{\ell}{k}\right)^\ell\|\nabla\tilde{u}_i^{far}\|_{L^2(\omega_{i,1})}.
\end{equation}

Similarly, choosing $p\leftarrow \ell$, $\ell\leftarrow k$, and  $c\leftarrow c_0$ in the second last estimate of \cite[p.\,172]{Boerm} we get
\begin{equation}\label{Boerm2}
\|\nabla(\tilde{u}_i^{far}-\hat{u}_i^{far})\|_{L^2(\omega_i)}\leq  \left(c_0\frac{\ell}{k}\right)^\ell\|\nabla\tilde{u}_
i^{far}\|_{L^2(\omega_{i,1})}.
\end{equation}
According to the definition of $\ell$ we have to distinguish the following two cases:
\begin{itemize}
\item Case 1: $\left\lceil\frac{2}{\log 2}\log\frac{1}{H_i}\right\rceil\leq2$\\
By definition of $\ell$ we know that $\ell=2$ and after some simple calculations we see that $H_i\geq\frac{1}{2}$. Therefore we obtain by the definition of $k$
\begin{equation}\label{case1}
\left(c_0\frac{\ell}{k}\right)^\ell=\left(\frac{\ell-1}{2\ell}\right)^\ell=\frac{1}{16}<\frac{1}{4}\leq H_i^{2}.
\end{equation}

\item Case 2: $\left\lceil\frac{2}{\log 2}\log\frac{1}{H_i}\right\rceil>2$\\
Set $\alpha:=\frac{2}{\log 2}$. Then $\ell=\lceil-\alpha\log H_i\rceil\geq-\alpha\log H_i$ and furthermore we have
\begin{align}\label{case2}
\left(c_0\frac{\ell}{k}\right)^\ell&=\left(\frac{\ell-1}{2\ell}\right)^\ell\leq 2^{-\ell}=e^{-\ell\log 2}\leq H_i^{\alpha\log 2}=H_i^{2}.
\end{align}
\end{itemize}
The assertion follows by combining \eqref{Boerm1}, \eqref{Boerm2}, \eqref{case1}, and \eqref{case2}.
\end{proof}

\begin{lemma}\label{lemma_v_near_far}
Define $d_i^{near}:=u_i^{near}-\tilde{u}_i^{near}$ and $d_i^{far}:=u_i^{far}-\hat{u}_i^{far}$ with $u_i^{near}$, $u_i^{far}$ as in \eqref{splitting}, $\tilde{u}_i^{near}$ as in \eqref{approx_near_far} and $\hat{u}_i^{far}$ as in Lemma \ref{lemma_aprox_tilde_u_far}. Set
\begin{equation*}
v^{near}:=\sum_{i=1}^Nb_id_i^{near}\qquad \text{and}\qquad v^{far}:=\sum_{i=1}^Nb_id_i^{far}.
\end{equation*}
Then the estimates
\begin{align*}
&\|\nabla v^{near}\|_{L^2(\Omega)}^2\leq 2M_0\sum\limits_{i=1}^N\left(\left\|\nabla d_i^{near}\right\|_{L^2\left(\omega_i\right)}^2+\frac{C^2}{H_i^2}\left\|d_i^{near}\right\|_{L^2\left(\omega_i\right)}^2\right)\\
&\|\nabla v^{far}\|_{L^2(\Omega)}^2\leq 2M_0\sum\limits_{i=1}^N\left(\left\|\nabla d_i^{far}\right\|_{L^2\left(\omega_i\right)}^2+\frac{C^2}{H_i^2}\left\|d_i^{far}\right\|_{L^2\left(\omega_i\right)}^2\right)
\end{align*}
hold with $M_0$ as in \eqref{overlap_const} and $H_i$ as in \eqref{local_mesh_width}.
\end{lemma}

\begin{proof}
Applying Cauchy--Schwarz inequality, using the Leibniz rule for products, a triangle inequality and an inverse inequality for $b_i$ we obtain the estimate
\begin{align*}
\left\|\nabla v^{near}\right\|_{L^2\left(\Omega\right)}^2&=\langle\nabla v^{near},\nabla v^{near}\rangle_{L^2\left(\Omega\right)}=\sum\limits_{i=1}^N\langle\nabla\left(b_id_i^{near}\right),\nabla v^{near}\rangle_{L^2\left(\Omega\right)}\\
&\leq \sum\limits_{i=1}^N\left\|\nabla\left(b_id_i^{near}\right)\right\|_{L^2\left(\omega_i\right)}\left\|\nabla v^{near}\right\|_{L^2\left(\omega_i\right)}\\
&\leq \sum\limits_{i=1}^N\left(\left\|\nabla d_i^{near}\right\|_{L^2\left(\omega_i\right)}+\frac{C}{H_i}\left\|d_i^{near}\right\|_{L^2\left(\omega_i\right)}\right)\left\|\nabla v^{near}\right\|_{L^2\left(\omega_i\right)}.
\end{align*}

A Young's inequality leads to
\begin{equation*}
\left\|\nabla v^{near}\right\|_{L^2\left(\Omega\right)}^2\leq \sum\limits_{i=1}^N\frac{\epsilon^2}{2}\left(\left\|\nabla d_i^{near}\right\|_{L^2\left(\omega_i\right)}+\frac{C}{H_i}\left\|d_i^{near}\right\|_{L^2\left(\omega_i\right)}\right)^2+\frac{1}{2\epsilon^2}\sum\limits_{i=1}^N\left\|\nabla v^{near}\right\|_{L^2\left(\omega_i\right)}^2.
\end{equation*}

The choice $\epsilon^2=M_0$ (cf.\ Remark \ref{rem_overlap}) yields
\begin{IEEEeqnarray}{rCl}\label{eps3}
\left\|\nabla v^{near}\right\|_{L^2\left(\Omega\right)}^2 &\leq & \sum\limits_{i=1}^N\frac{M_0}{2}\left(\left\|\nabla d_i^{near}\right\|_{L^2\left(\omega_i\right)}+\frac{C}{H_i}\left\|d_i^{near}\right\|_{L^2\left(\omega_i\right)}\right)^2
\nonumber\\
&& +\frac{1}{2M_0}\sum\limits_{i=1}^n\left\|\nabla v^{near}\right\|_{L^2\left(\omega_i\right)}^2
\nonumber\\
&\leq & \sum\limits_{i=1}^N\frac{M_0}{2}\left(\left\|\nabla d_i^{near}\right\|_{L^2\left(\omega_i\right)}+\frac{C}{H_i}\left\|d_i^{near}\right\|_{L^2\left(\omega_i\right)}\right)^2
\nonumber\\
&& +\frac{1}{2}\left\|\nabla v^{near}\right\|_{L^2\left(\Omega\right)}^2.
\end{IEEEeqnarray}

Hence, by \eqref{eps3} and a triangle inequality we get 
\begin{align*}
\left\|\nabla v^{near}\right\|_{L^2\left(\Omega\right)}^2&\leq M_0\sum\limits_{i=1}^N\left(\left\|\nabla d_i^{near}\right\|_{L^2\left(\omega_i\right)}+\frac{C}{H_i}\left\|d_i^{near}\right\|_{L^2\left(\omega_i\right)}\right)^2\nonumber\\
&\leq 2M_0\sum\limits_{i=1}^N\left(\left\|\nabla d_i^{near}\right\|_{L^2\left(\omega_i\right)}^2+\frac{C^2}{H_i^2}\left\|d_i^{near}\right\|_{L^2\left(\omega_i\right)}^2\right).
\end{align*}
This shows the first estimate. The proof of the second estimate is verbatim the same.
\end{proof}

\begin{lemma}\label{lemma_est_d_near}
Let $d_i^{near}$ as in Lemma \ref{lemma_v_near_far}. If Assumption \ref{ass} holds, then 
\begin{equation*}
\|\nabla d_i^{near}\|_{L^2(\omega_{i})}\leq C\left( H_i^2+H_i\right)\left\|f\right\|_{L^2(\omega_{i,1})}
\end{equation*}
and
\begin{equation*}
\| d_i^{near}\|_{L^2(\omega_i)}\leq C\left(H_i^3+H_i^2\right)\left\|f\right\|_{L^2(\omega_{i,1})}
\end{equation*}
with $H_i$ as in \eqref{local_mesh_width} and constants $C$ which depend on $\alpha$, $\beta$ (cf.\ \eqref{coeff}).
\end{lemma}

\begin{proof}
\eqref{splitting}, \eqref{approx_near_far} and a triangle inequality yield
\begin{IEEEeqnarray}{rCl}\label{nabla_d_near}
\left\|\nabla d_i^{near}\right\|_{L^2(\omega_{i,2})}& = & \left\|\nabla\left(L_{\omega_{i,2}}^{-1}(g_i^{near})-\tilde{L}_{\omega_{i,2}}^{-1}(P_ig_i^{near})\right)\right\|_{L^2(\omega_{i,2})}
\nonumber\\
&\leq & \left\|\nabla\left(L_{\omega_{i,2}}^{-1}(g_i^{near})-\tilde{L}_{\omega_{i,2}}^{-1}(g_i^{near})\right)\right\|_{L^2(\omega_{i,2})}
\nonumber\\
&& +\left\|\nabla\left(\tilde{L}_{\omega_{i,2}}^{-1}(g_i^{near}-P_ig_i^{near})\right)\right\|_{L^2(\omega_{i,2})},
\end{IEEEeqnarray}
where $P_i$ is the $L^2$-orthogonal projection of $L^2(\omega_{i,1})$ onto $S_0(\mathcal{G}_{i,1})$.
In order to estimate the second term of \eqref{nabla_d_near} we use \eqref{L} and Lemma \ref{lemma_projection}. This leads to
\begin{align}\label{d_near_nabla_second}
\left\|\nabla\left(\tilde{L}_{\omega_{i,2}}^{-1}(g_i^{near}-P_ig_i^{near})\right)\right\|_{L^2(\omega_{i,2})}&\leq\frac{C}{\alpha}\|g_i^{near}-P_ig_i^{near}\|_{H^{-1}(\omega_{i,2})}\nonumber\\
&\leq C\frac{H_i}{\alpha}\|g_i^{near}\|_{L^2(\omega_{i,2})}.
\end{align} 

By \eqref{nabla_d_near}, \eqref{assumption}, \eqref{d_near_nabla_second} and the definition of $g_i^{near}$ (cf.\ \eqref{g_i_near}) we obtain

\begin{align*}
\left\|\nabla d_i^{near}\right\|_{L^2(\omega_{i,2})}&\leq \frac{\beta}{\alpha}CH_i^2\left\|g_i^{near}\right\|_{L^2(\omega_{i,2})}+C\frac{H_i}{\alpha}\left\|g_i^{near}\right\|_{L^2(\omega_{i,2})}\\
&\leq C\left( H_i^2+H_i\right)\left\|f\right\|_{L^2(\omega_{i,1})}.
\end{align*}

Since $\omega_{i}\subset\omega_{i,2}$ we also have
\begin{equation*}
\left\|\nabla d_i^{near}\right\|_{L^2(\omega_{i})}\leq C\left( H_i^2+H_i\right)\left\|f\right\|_{L^2(\omega_{i,1})} .
\end{equation*}

By \eqref{splitting}, \eqref{approx_near_far}, a triangle inequality and Friedrichs' inequality we get
\begin{IEEEeqnarray}{rCl}\label{d_near}
\left\|d_i^{near}\right\|_{L^2(\omega_i)}&=&\left\|L_{\omega_{i,2}}^{-1}(g_i^{near})-\tilde{L}_{\omega_{i,2}}^{-1}(P_ig_i^{near})\right\|_{L^2(\omega_i)}
\nonumber\\
&\leq& \left\|L_{\omega_{i,2}}^{-1}(g_i^{near})-\tilde{L}_{\omega_{i,2}}^{-1}(g_i^{near})\right\|_{L^2(\omega_{i})}+\left\| \tilde{L}_{\omega_{i,2}}^{-1}(g_i^{near}-P_ig_i^{near})\right\|_{L^2(\omega_i)}
\nonumber\\
&\leq & CH_i \Big(\left\|\nabla\left(L_{\omega_{i,2}}^{-1}(g_i^{near})-\tilde{L}_{\omega_{i,2}}^{-1}(g_i^{near})\right)\right\|_{L^2(\omega_{i,2})}
\nonumber\\
&&+\left\|\nabla\left(\tilde{L}_{\omega_{i,2}}^{-1}(g_i^{near}-P_ig_i^{near})\right)\right\|_{L^2(\omega_{i,2})}\Big).
\end{IEEEeqnarray}
 The combination of \eqref{d_near}, \eqref{assumption} and \eqref{d_near_nabla_second} leads to
\begin{align*}
\left\|d_i^{near}\right\|_{L^2(\omega_i)}&\leq  C\frac{\beta}{\alpha}H_i^3\left\|g_i^{near}\right\|_{L^2(\omega_{i,2})}+C\frac{H_i^2}{\alpha}\left\|g_i^{near}\right\|_{L^2(\omega_{i,2})}
\\
&\leq C\left(H_i^3+H_i^2\right)\left\|f\right\|_{L^2(\omega_{i,1})}.
\end{align*}
In the last step we used the definition of $g_i^{near}$ (cf.\ \eqref{g_i_near}).
\end{proof}

\begin{lemma}\label{lemma_est_d_far}
Let $d_i^{far}$ as in Lemma \ref{lemma_v_near_far}. If Assumption \ref{ass} holds, then 
\begin{equation*}
\|\nabla d_i^{far}\|_{L^2(\omega_{i})}\leq C \left(H_i^2+h_i\right)\|g_i^{far}\|_{L^2(\omega_i^{far})}
\end{equation*}
and
\begin{equation*}
\| d_i^{far}\|_{L^2(\omega_i)}\leq C\left(H_i^3+H_ih_i\right)\|g_i^{far}\|_{L^2(\omega_i^{far})}
\end{equation*}
with $H_i$ as in \eqref{local_mesh_width} and $h_i:=\max_{\tau \in\mathcal{R}^t(\mathcal{G}_i^{far})}\diam\tau$ is the mesh width of the refined mesh $\mathcal{R}^t(\mathcal{G}_i^{far})$ (cf.\ \eqref{omega_far} and \eqref{refinement_op}). The constants $C$ depend on $\alpha$, $\beta$ (cf.\ \eqref{coeff}).
\end{lemma}

\begin{proof}
By the definition of $d_i^{far}$, \eqref{splitting} and two  triangle inequalities we get
\begin{IEEEeqnarray}{rCl}\label{nabla_d_far}
\|\nabla d_i^{far}\|_{L^2(\omega_i)}&=&\left\|\nabla\left(L_{\omega_{i,2}}^{-1}(g_i^{far})-\hat{u}_i^{far}\right)\right\|_{L^2(\omega_i)}
\nonumber\\
&\leq & \left\|\nabla\left(L_{\omega_{i,2}}^{-1}(g_i^{far})-\tilde{L}_{\omega_{i,2}}^{-1}(P_i^tg_i^{far})\right)\right\|_{L^2(\omega_{i})}+\|\nabla(\tilde{u}_i^{far}-\hat{u}_i^{far})\|_{L^2(\omega_i)}
\nonumber\\
&\leq &\left\|\nabla\left(L_{\omega_{i,2}}^{-1}(g_i^{far})-\tilde{L}_{\omega_{i,2}}^{-1}(g_i^{far})\right)\right\|_{L^2(\omega_{i})}+\left\|\nabla\left(\tilde{L}_{\omega_{i,2}}^{-1}(g_i^{far}-P_i^tg_i^{far})\right)\right\|_{L^2(\omega_{i})}
\nonumber\\
&&+\|\nabla(\tilde{u}_i^{far}-\hat{u}_i^{far})\|_{L^2(\omega_i)},
\end{IEEEeqnarray}

where $P_i^t$ denotes the $L^2$-orthogonal projection of $L^2(\omega_i^{far})$ onto $S_0(\mathcal{R}^t(\mathcal{G}_i^{far}))$ and $\tilde{u}_i^{far}$ is as in \eqref{approx_near_far}.

For the first term of \eqref{nabla_d_far} we can use that $\omega_i\subset\omega_{i,2}$ and \eqref{assumption}. This leads to
\begin{align}\label{first_term_nabla_d_far}
\left\|\nabla\left(L_{\omega_{i,2}}^{-1}(g_i^{far})-\tilde{L}_{\omega_{i,2}}^{-1}(g_i^{far})\right)\right\|_{L^2(\omega_{i})}&\leq \left\|\nabla\left(L_{\omega_{i,2}}^{-1}(g_i^{far})-\tilde{L}_{\omega_{i,2}}^{-1}(g_i^{far})\right)\right\|_{L^2(\omega_{i,2})}\nonumber\\
&\leq \frac{\beta}{\alpha}CH_i^2\|g_i^{far}\|_{L^2(\omega_{i,2})}.
\end{align}

In order to get an estimate for the second term of \eqref{nabla_d_far} we use $\omega_i\subset\omega_{i,2}$, \eqref{L} and Lemma \ref{lemma_projection}. This yields
\begin{align}\label{second_term_nabla_d_far}
\left\|\nabla\left(\tilde{L}_{\omega_{i,2}}^{-1}(g_i^{far}-P_i^tg_i^{far})\right)\right\|_{L^2(\omega_{i})}&\leq \left\|\nabla\left(\tilde{L}_{\omega_{i,2}}^{-1}(g_i^{far}-P_i^tg_i^{far})\right)\right\|_{L^2(\omega_{i,2})}\nonumber\\
&\leq \frac{C}{\alpha}\|g_i^{far}-P_i^tg_i^{far}\|_{H^{-1}(\omega_{i,2})}\nonumber\\
&\leq C\frac{h_i}{\alpha}\|g_i^{far}\|_{L^2(\omega_{i,2})}.
\end{align}

The third term of \eqref{nabla_d_far} can be estimated by Lemma \ref{lemma_aprox_tilde_u_far}, \eqref{approx_near_far}, using that $\omega_{i,1}\subset\omega_{i,2}$, \eqref{L} and Friedrichs' inequality. Thus we have
\begin{align}\label{third_term_nabla_d_far}
\|\nabla(\tilde{u}_i^{far}-\hat{u}_i^{far})\|_{L^2(\omega_i)}&\leq CH_i^{2}\|\nabla\tilde{u}_i^{far}\|_{L^2(\omega_{i,1})}\nonumber\\
&\leq CH_i^{2}\left\|\nabla \tilde{L}_{\omega_{i,2}}^{-1}(P_i^tg_i^{far})\right\|_{L^2(\omega_{i,2})}\nonumber\\
&\leq\frac{C}{\alpha}H_i^{2}\|P_i^tg_i^{far}\|_{H^{-1}(\omega_{i,2})}\nonumber\\
&\leq\frac{C}{\alpha}H_i^{2}\|g_i^{far}\|_{L^2(\omega_{i,2})}.
\end{align}

Hence, the combination of \eqref{nabla_d_far}, \eqref{first_term_nabla_d_far}, \eqref{second_term_nabla_d_far}, \eqref{third_term_nabla_d_far} and recalling that $g_i^{far}|_{\omega_{i,1}}=0$ yields
\begin{align*}
\|\nabla d_i^{far}\|_{L^2(\omega_i)}&\leq \left(\frac{\beta}{\alpha}CH_i^2+C\frac{h_i}{\alpha}+\frac{C}{\alpha}H_i^{2}\right)\|g_i^{far}\|_{L^2(\omega_i^{far})}\\
&\leq C(H_i^2+h_i)\|g_i^{far}\|_{L^2(\omega_i^{far})}.
\end{align*}

The estimate for the $L^2$-norm of $d_i^{far}$ can be obtained similarly. By triangle inequalities, Friedrichs' inequality, Lemma \ref{lemma_aprox_tilde_u_far}, \eqref{assumption}, \eqref{L} and Lemma \ref{lemma_projection} we get
\begin{IEEEeqnarray*}{rCl}
\|d_i^{far}\|_{L^2(\omega_i)}&=&\left\|L_{\omega_{i,2}}^{-1}(g_i^{far})-\hat{u}_i^{far}\right\|_{L^2(\omega_i)}
\nonumber\\
&\leq & \left\|L_{\omega_{i,2}}^{-1}(g_i^{far})-\tilde{L}_{\omega_{i,2}}^{-1}(g_i^{far})\right\|_{L^2(\omega_i)}+\left\| \tilde{L}_{\omega_{i,2}}^{-1}(g_i^{far}-P_i^tg_i^{far})\right\|_{L^2(\omega_i)}
\nonumber\\
&&+\|\tilde{u}_i^{far}-\hat{u}_i^{far}\|_{L^2(\omega_i)}
\nonumber\\
&\leq & CH_i\left\|\nabla\left(L_{\omega_{i,2}}^{-1}(g_i^{far})-\tilde{L}_{\omega_{i,2}}^{-1}(g_i^{far})\right)\right\|_{L^2(\omega_{i,2})}
\nonumber\\
&&+CH_i\left\|\nabla\left(\tilde{L}_{\omega_{i,2}}^{-1}(g_i^{far}-P_i^tg_i^{far})\right)\right\|_{L^2(\omega_{i,2})}
+CH_i^{3}\|\nabla\tilde{u}_i^{far}\|_{L^2(\omega_{i,1})}
\nonumber\\
&\leq & \left(C\frac{\beta}{\alpha}H_i^3+CH_i\frac{h_i}{\alpha}+\frac{C}{\alpha}H_i^{3}\right)\|g_i^{far}\|_{L^2(\omega_i^{far})}
\nonumber\\
&\leq & C(H_i^3+H_ih_i)\|g_i^{far}\|_{L^2(\omega_i^{far})}.
\end{IEEEeqnarray*}
\end{proof}

\begin{theorem}\label{main_theo}
Let $\Omega\subset\R^d$ ($d\geq 2$) be a bounded domain with $\partial\Omega\in C^1$ and let Assumption \ref{ass} be satisfied. Let $u$ denote the solution of \eqref{prob_weak} and $u_{AL}^{GAL}$ its approximation given by \eqref{Galerkin_sol_II}. Let the parameters $\ell$ and $k$ in the definition of the farfield part of $V_{AL}$ be chosen according to
\begin{equation*}
\ell:=\max \left\{2,\left\lceil\frac{2}{\log 2}\log\frac{1}{H_i}\right\rceil\right\}\quad \text{and}\quad k:=\left\lceil\frac{2c_0\ell^2}{(\ell-1)}\right\rceil
\end{equation*}
for some $c_0=O(1)$. Moreover let $Q\in (6,\infty)$ and $P\in(2Q/(Q-6),\infty)$ be fixed. Assume that $A$ satisfies \eqref{coeff} as well as $\alpha/\beta\in[\max\{1-1/K_Q,1-1/K_P\},1]$ with $K_Q$ and $K_P$ as in Theorem \ref{theo_Simader}. Further let $f\in L^{P}(\Omega)$ and assume that there exists a constant $C>0$ such that $N\leq CH^{-d}$ holds. If the refinement parameter $t$ is chosen according to 
\begin{equation*}
t:=\left\lceil\lb\frac{1}{H}\right\rceil,
\end{equation*} 
then the error estimate
\begin{equation}\label{main_estimate}
\left\|A^{1/2}\nabla(u-u_{AL}^{GAL})\right\|_{L^2(\Omega)}\leq CH\left\|f\right\|_{L^{p}(\Omega)}
\end{equation}
holds for any $p\in (2Q/(Q-6),P]$ with $p=2q/(q-2)$ for some $2<q< \frac{Q}{3}$. The constant $C$ depends on $\alpha$, $\beta$ and $p$.\\
 For the dimension we have
\begin{equation}\label{dimension_tilde_VAL-II}
\dim V_{AL}\leq CN\ell^{d+1}\leq CH^{-d}\log^{d+1}\frac{1}{H}.
\end{equation}
\end{theorem}

\begin{proof}
Let $f\in L^P(\Omega)$ and set $u:=L_\Omega^{-1}f$. Let $u_{AL}^{GAL}\in V_{AL}$ be the Galerkin approximation of $u$ given by \eqref{Galerkin_sol_II}. By the Galerkin orthogonality we obtain for any $u_{AL}\in V_{AL}$
\begin{align*}
\|A^{1/2}\nabla(u-u_{AL}^{GAL})\|_{L^2(\Omega)}^2&=a(u-u_{AL}^{GAL},u-u_{AL}^{GAL})\\&=a(u-u_{AL}^{GAL},u-u_{AL})\\&\leq \|A^{1/2}\nabla(u-u_{AL}^{GAL})\|_{L^2(\Omega)}\|A^{1/2}\nabla(u-u_{AL})\|_{L^2(\Omega)}.
\end{align*}
Hence,
\begin{equation}\label{Galerkin_ortho}
\|A^{1/2}\nabla(u-u_{AL}^{GAL})\|_{L^2(\Omega)}\leq \|A^{1/2}\nabla(u-u_{AL})\|_{L^2(\Omega)}\quad\forall\; u_{AL}\in V_{AL}.
\end{equation}
Further let $u_i^{near}$ and $u_i^{far}$ as in \eqref{splitting}. Then it holds that
\begin{equation*}
u=\sum_{i=1}^N b_i(u_i^{near}+u_i^{far}).
\end{equation*}
Let  $\tilde{u}_i^{near}$  and $\tilde{u}_i^{far}$ as in \eqref{approx_near_far}. We choose $\hat{u}_i^{far}$  as in Lemma \ref{lemma_aprox_tilde_u_far} and $u_{AL}\in V_{AL}$ by
\begin{equation*}
u_{AL}=\sum_{i=1}^N b_i(\tilde{u}_i^{near}+\hat{u}_i^{far}).
\end{equation*}

Using this notation we have
\begin{equation*}
u-u_{AL}=\sum_{i=1}^Nb_i(u_i^{near}-\tilde{u}_i^{near})+\sum_{i=1}^Nb_i(u_i^{far}-\hat{u}_i^{far}).
\end{equation*}

First we consider the nearfield part. Let $d_i^{near}:=u_i^{near}-\tilde{u}_i^{near}$ and set
\begin{equation*}
v^{near}:=\sum_{i=1}^Nb_id_i^{near}.
\end{equation*}

By Lemma \ref{lemma_v_near_far} we know that
\begin{equation*}
\left\|\nabla v^{near}\right\|_{L^2\left(\Omega\right)}^2
\leq 2M_0 \sum\limits_{i=1}^N\left(\left\|\nabla d_i^{near}\right\|_{L^2\left(\omega_i\right)}^2+\frac{C^2}{H_i^2}\left\|d_i^{near}\right\|_{L^2\left(\omega_i\right)}^2\right).
\end{equation*}

Moreover, by Lemma \ref{lemma_est_d_near} and since every simplex $\tau$ is contained in at most $M_1$ domains $\omega_{i,1}$ (cf.\ \eqref{overlap_const}) we obtain
\begin{align*}
\left\|\nabla v^{near}\right\|_{L^2(\Omega)}&\leq C\sqrt{\sum_{i=1}^N\left(H_i^4+H_i^2\right)\left\|f\right\|_{L^2(\omega_{i,1})}^2}\\
&\leq C\left(H^2+H\right)\left\|f\right\|_{L^2(\Omega)}.
\end{align*}

Since the embedding $L^p(\Omega)\hookrightarrow L^2(\Omega)$ is continuous for any $p\geq 2$, we have
\begin{equation}\label{estimate_nearfield}
\left\|\nabla v^{near}\right\|_{L^2(\Omega)}\leq C\left(H^2+H\right)\left\|f\right\|_{L^{p}(\Omega)}\leq CH\|f\|_{L^p(\Omega)}
\end{equation}
for any $p\geq 2$.

Next we consider the farfield part. Let $d_i^{far}:=u_i^{far}-\hat{u}_i^{far}$ and set
\begin{equation*}
v^{far}:=\sum_{i=1}^Nb_id_i^{far}.
\end{equation*}

Lemma \ref{lemma_v_near_far} yields
\begin{equation*}
\|\nabla v^{far}\|_{L^2(\Omega)}^2\leq 2M_0\sum\limits_{i=1}^N\left(\|\nabla d_i^{far}\|_{L^2(\omega_i)}^2+\frac{C^2}{H_i^2}\|d_i^{far}\|_{L^2(\omega_i)}^2\right).
\end{equation*}

Due to Lemma \ref{lemma_est_d_far} we finally get with a constant $C$ depending on the mesh regularity
\begin{equation}\label{nabla_vfar}
\|\nabla v^{far}\|_{L^2(\Omega)}\leq C\sqrt{\sum_{i=1}^N\left(H_i^4+h_i^2\right)\|g_i^{far}\|_{L^2(\omega_i^{far})}^2}.
\end{equation}

By the definition of $g_i^{far}$ (cf.\ \eqref{g_i_far}) we have
\enlargethispage{-3\baselineskip} % quirk fuer formel, damit abschaetzung auf gleicher seite ist
\begin{IEEEeqnarray}{rCl}\label{gifar}
\|g_i^{far}\|_{L^2(\omega_i^{far})}&\leq & \|\chi_if\|_{L^2(\omega_i^{far})}+2\|A\nabla \chi_i\nabla u\|_{L^2(\omega_i^{far})}
\nonumber\\
&&+\|(u-\bar{u}_i)\Div(A\nabla\chi_i)\|_{L^2(\omega_i^{far})}.
\end{IEEEeqnarray}

Applying general H\"{o}lder's  inequality on the first term of \eqref{gifar} and by \eqref{cutoff_estimate_1} we obtain for any $2< q< Q/3$ and any $p\in(2Q/(Q-6),P]$ such that $2/q+2/p=1$ the estimate
\begin{align}\label{first_term_g_i}
\|\chi_if\|_{L^2(\omega_i^{far})}&\leq\|\chi_i\|_{L^{q}(\omega_i^{far})}\|f\|_{L^{p}(\omega_i^{far})}\nonumber\\
&\leq CH_i^{\frac{d}{q}}\|f\|_{L^{p}(\omega_i^{far})}\nonumber\\
&=CH_i^{\frac{d}{2}-\frac{d}{p}}\|f\|_{L^{p}(\omega_i^{far})}.
\end{align}

To get an estimate of the second term of \eqref{gifar} we use general H\"{o}lder's inequality, \eqref{coeff} and \eqref{cutoff_estimate_2}. For $2< q< Q/3$ and any $p\in(2Q/(Q-6),P]$ such that $2/q+2/p=1$ it holds
\begin{align}\label{second_term_g_i}
\|A\nabla\chi_i\nabla u\|_{L^2(\omega_i^{far})}&\leq \|A\|_{L^\infty(\omega_i^{far})}\|\nabla\chi_i\|_{L^{q}(\omega_i^{far})}\|\nabla u\|_{L^{p}(\omega_i^{far})}\nonumber\\
&\leq C\beta H_i^{\frac{d}{2}-\frac{d}{p}-1}\|\nabla u\|_{L^{p}(\omega_i^{far})}.
\end{align}

For the third term of  \eqref{gifar} we obtain by general H\"{o}lder's inequality, using \eqref{cutoff_estimate_3} and by Poincar\'{e}'s inequality for $2< q< Q/3$ and any $p\in(2Q/(Q-6),P]$ such that $2/q+2/p=1$
\begin{align}\label{third_term_g_i}
\|(u-\bar{u}_i)\Div(A\nabla\chi_i)\|_{L^2(\omega_i^{far})}&\leq \|\Div(A\nabla\chi_i)\|_{L^{q}(\omega_i^{far})}\|u-\bar{u}_i\|_{L^p(\omega_i^{far})}\nonumber\\
&\leq CH_i^{\frac{d}{2}-\frac{d}{p}-1}\|\nabla u\|_{L^{p}(\omega_i^{far})}.
\end{align}

Next, we want to estimate the square root of $\sum_{i=1}^N (H_i^4+h_i^2)\|\chi_i f\|_{L^2(\omega_i^{far})}^2$. For this we set $\gamma_i:=(H_i^4+h_i^2)H_i^{d-2d/p}$ and $\delta_i:=\|f\|_{L^{p}(\omega_i^{far})}^2$. By \eqref{first_term_g_i} we get

\begin{align*}
\sqrt{\sum\limits_{i=1}^N (H_i^4+h_i^2)\|\chi_i f\|_{L^2(\omega_i^{far})}^2}&\leq C\sqrt{\sum\limits_{i=1}^N(H_i^4+h_i^2)H_i^{d-\frac{2d}{p}}\|f\|_{L^{p}(\omega_i^{far})}^2}\\
&= C\sqrt{\sum\limits_{i=1}^N \gamma_i\delta_i}.
\end{align*}

Applying a discrete H\"{o}lder's inequality with $r:=p/2$ and $r'=p/(p-2)$ yields
\begin{align*}
\sqrt{\sum\limits_{i=1}^N (H_i^4+ h_i^2)\|\chi_i f\|_{L^2(\omega_i^{far})}^2}&\leq C \sqrt{\|\gamma_i\|_{\ell^{r'}} \|\delta_i\|_{\ell^r}}\\
&=C\sqrt{\left(\sum\limits_{i=1}^N \gamma_i^{r'}\right)^{\frac{1}{{r'}}}\left(\sum\limits_{i=1}^N \delta_i^{r}\right)^{\frac{1}{r}}}\\
&=C\left(\sum_{i=1}^N (H_i^4+h_i^2)^{\frac{p}{p-2}}H_i^{d}\right)^{\frac{p-2}{2p}}\left(\sum\limits_{i=1}^N\|f\|_{L^p(\omega_i^{far})}^{p}\right)^{\frac{1}{p}}.
\end{align*}

Since $\omega_i^{far}\subset \omega_{i,2}$ and every simplex $\tau$ is contained in at most $M_2$ domains $\omega_{i,2}$ (cf.\ \eqref{overlap_const}) we obtain
\begin{align}\label{first_term_sum}
\sqrt{\sum\limits_{i=1}^N (H_i^4+h_i^2)\|\chi_i f\|_{L^2(\omega_i^{far})}^2}&\leq C\left(N \max_{1\leq i\leq N}(H_i^4+h_i^2)^{\frac{p}{p-2}}H_i^{d}\right)^{\frac{p-2}{2p}} \left(\sum\limits_{i=1}^N\|f\|_{L^{p}(\omega_{i,2})}^{p}\right)^{\frac{1}{p}}\nonumber\\
&\leq C (H^2+h)\|f\|_{L^{p}(\Omega)}.
\end{align}
The last inequality follows due to the assumption that  $N\leq CH^{-d}$.

Now, we want to estimate the square root of $\sum_{i=1}^N (H_i^4+h_i^2)\|A\nabla\chi_i\nabla u\|_{L^2(\omega_i^{far})}^2$ in a similar way.
By \eqref{second_term_g_i} and using a discrete H\"{o}lder's inequality with $r, r', \gamma_i$ as before and $\delta_i:=H_i^{-2}\|\nabla u\|_{L^{p}(\omega_i^{far})}^2$ we get 
\begin{align*}
\sqrt{\sum\limits_{i=1}^N (H_i^4+h_i^2)\|A\nabla\chi_i\nabla u\|_{L^2(\omega_i^{far})}^2}&\leq C\beta\sqrt{\sum\limits_{i=1}^N(H_i^4+h_i^2)H_i^{d-\frac{2d}{p}-2}\|\nabla u\|_{L^{p}(\omega_i^{far})}^2} \\
&\leq C\beta \sqrt{\|\gamma_i\|_{\ell^{r'}} \|\delta_i\|_{\ell^r}}\\
&\leq C\beta (H^2+h)\left(\sum\limits_{i=1}^N H_i^{-p}\|\nabla u\|_{L^{p}(\omega_i^{far})}^{p}\right)^{\frac{1}{p}}\\
&\leq C\beta (H^2+h)\left(\max_{1\leq i\leq N}H_i^{-p}\sum\limits_{i=1}^N\|\nabla u\|_{L^{p}(\omega_{i,2})}^{p}\right)^{\frac{1}{p}}\\
&\leq C\beta \left(H+\frac{h}{H}\right)\|\nabla u\|_{L^{p}(\Omega)}.
\end{align*}

By Theorem \ref{theo_gradient_estimate} we obtain\footnote{Note that \begin{align*}
\left\|F\right\|_{W^{-1,p}(\Omega)}&=\sup_{\substack{v\in W_0^{1,p'}(\Omega)\\ \|v\|W^{1,p'}(\Omega) \leq 1}}\left|\int_\Omega\! fv\right|
\leq \sup_{\substack{v\in W_0^{1,p'}(\Omega)\\ \|v\|W^{1,p'}(\Omega)\leq 1}}\left\|f\right\|_{L^p(\Omega)}\left\|v\right\|_{L^{p'}(\Omega)}
\leq \left\|f\right\|_{L^p(\Omega)}.
\end{align*}}
\enlargethispage{-3\baselineskip}
\begin{align}\label{second_term_sum}
\sqrt{\sum\limits_{i=1}^N \left(H_i^4+h_i^2\right)\|A\nabla\chi_i\nabla u\|_{L^2(\omega_i^{far})}^2}&\leq C\beta \left(H+ \frac{h}{H}\right)\|F\|_{W^{-1,p}(\Omega)}\nonumber\\
&\leq C\beta \left(H+\frac{h}{H}\right)\|f\|_{L^{p}(\Omega)}.
\end{align}

Estimate \eqref{third_term_g_i} and the same arguments as above yield
\begin{align}\label{third_term_sum}
\sqrt{\sum\limits_{i=1}^N (H_i^4+h_i^2) \|(u-\bar{u}_i)\Div(A\nabla\chi_i)\|_{L^2(\omega_i^{far})}^2}&\leq C \sqrt{\sum\limits_{i=1}^N(H_i^4+h_i^2)H_i^{d-\frac{2d}{p}-2}\|\nabla u\|_{L^{p}(\omega_i^{far})}^2}\nonumber\\
&\leq C \left(H+\frac{h}{H}\right)\|f\|_{L^{p}(\Omega)}.
\end{align}

The combination of \eqref{nabla_vfar}, \eqref{gifar}, \eqref{first_term_sum}, \eqref{second_term_sum} and \eqref{third_term_sum} yields
\begin{equation*}
\|\nabla v^{far}\|_{L^2(\Omega)}\leq C\left(H^2+H+h+\frac{h}{H}\right)\|f\|_{L^{p}(\Omega)}
\end{equation*}
for any $p\in(2Q/(Q-6),P]$ such that $2/q+2/p=1$ for some $2<q< Q/3$. The constant $C$ depends on $\alpha$, $\beta$ and $p$.
The small mesh size $h$ arises by $t$-fold refinement of the local coarse grid so that 
\begin{equation}\label{choice_t_2d}
h\leq CH2^{-t}. 
\end{equation}
By choosing 
\begin{equation*}
t=\left\lceil\lb\frac{1}{H}\right\rceil
\end{equation*}
in \eqref{choice_t_2d} $h$ satisfies
\begin{equation*}
h\leq CH^2.
\end{equation*}
Thus it holds
\begin{equation}\label{estimate_farfield}
\|\nabla v^{far}\|_{L^2(\Omega)}\leq C\left(H^2+H+h+\frac{h}{H}\right)\|f\|_{L^{p}(\Omega)}\leq CH\|f\|_{L^{p}(\Omega)}.
\end{equation}

The combination of \eqref{Galerkin_ortho}, \eqref{estimate_nearfield} and \eqref{estimate_farfield} leads to
\begin{align*}
\|A^{1/2}\nabla(u-u_{AL}^{GAL})\|_{L^2(\Omega)}&\leq \|A^{1/2}\nabla(u-u_{AL})\|_{L^2(\Omega)}\\
&=\|A^{1/2}(\nabla v^{near}+\nabla v^{far})\|_{L^2(\Omega)}\\
&\leq \|A^{1/2}\|_{L^\infty(\Omega)}\left(\|\nabla v^{near}\|_{L^2(\Omega)}+\|\nabla v^{far}\|_{L^2(\Omega)}\right)\\
&\leq \sqrt{\beta} C H \|f\|_{L^p(\Omega)}
\end{align*}

which finally proves estimate \eqref{main_estimate}. Estimate \eqref{dimension_tilde_VAL-II} can be seen as follows: From the definition of $V_{AL}$ (cf.\ \eqref{def:VAL-II}) it is clear that 
\begin{equation*}
\dim V_{AL}\leq N(\dim V_i^{near}+\dim V_i^{far})
\end{equation*}
holds. The choice  $k:=\left\lceil\frac{2c_0\ell^2}{(\ell-1)}\right\rceil$ yields
\begin{align*}
k&\leq \frac{2c_0\ell^2}{\ell -1}+1\\
&=\frac{2c_0\ell(\ell -1)}{\ell-1}+\frac{2c_0(\ell-1)}{\ell-1}+\frac{2c_0}{\ell-1}+1\\
&=2c_0\ell+2c_0+\frac{2c_0}{\ell-1}+1\\
&\leq 2c_0\ell+4c_0+1\\
&\leq \ell\left(4c_0+\frac{1}{2}\right).
\end{align*}
In the last inequality we used that $\ell\geq 2$. Remark \ref{remdim} and the above computation show that 
\begin{equation*}
\dim V_i^{far}\leq \ell k^d\leq \left(4c_0+\frac{1}{2}\right)^d\ell^{d+1}.
\end{equation*}
Obviously we have $\dim V_i^{near}=O(1)$. Hence,
\begin{equation*}
\dim V_{AL}\leq CN\ell^{2}\leq CH^{-d}\log^{d+1}\frac{1}{H}.
\end{equation*}
The last inequality follows by the assumption that there exists a constant $C>0$ such that $N\leq CH^{-d}$ and the choice of $\ell$.

\end{proof}

\bibliographystyle{plain}
\bibliography{mybib}{}

\section*{Acknowledgment}
I would like to thank Prof. Dr. Stefan Sauter for many interesting and helpful discussions. 

\end{document}